\documentclass{article}

\usepackage[a4paper]{geometry}
\usepackage{amsmath}
\usepackage{amssymb}
\usepackage{amsthm}
\usepackage{caption2}
\usepackage{multirow}

 \newtheorem{thm}{Theorem}[section]

\newtheorem{lem}[thm]{Lemma}
 \newtheorem{prop}[thm]{Proposition}
\newtheorem{cor}[thm]{Corollary}

\newcommand{\A}{{\mathcal A}}

\title{A strengthened inequality of Alon-Babai-Suzuki's conjecture on set systems with restricted intersections modulo $p$}

\author{Wang Xin$^{\text{a}}$, Hengjia Wei$^{\text{b}}$ and Gennian Ge$^{\text{b,c,}}$\thanks{Corresponding author. Email address:  11235062@zju.edu.cn (X. Wang), ven0505@163.com (H. Wei), gnge@zju.edu.cn (G. Ge).}\\
  \footnotesize $^{\text{a}}$ School of Mathematical Sciences, Zhejiang University, Hangzhou 310027, Zhejiang, China\\
  \footnotesize $^{\text{b}}$ School of Mathematical Sciences, Capital Normal University, Beijing, 100048, China\\
\footnotesize $^{\text{c}}$ Beijing Center for Mathematics and Information Interdisciplinary Sciences, Beijing, 100048, China}

\begin{document}
\date{}
\maketitle

\begin{abstract}
Let $K=\{k_1,k_2,\ldots,k_r\}$ and $L=\{l_1,l_2,\ldots,l_s\}$ be disjoint subsets of $\{0,1,\ldots,p-1\}$, where $p$ is a prime and $\A=\{A_1,A_2,\ldots,A_m\}$ be a family of subsets of $[n]$ such that $|A_i|\pmod{p}\in K$ for all $A_i\in \A$ and $|A_i\cap A_j|\pmod{p}\in L$ for $i\ne j$. In 1991, Alon, Babai and Suzuki conjectured that if $n\geq s+\max_{1\leq i\leq r} k_i$, then $|\A|\leq {n\choose s}+{n\choose s-1}+\cdots+{n\choose s-r+1}$. In 2000, Qian and Ray-Chaudhuri proved the conjecture under the condition $n\geq 2s-r$. In 2015, Hwang and Kim verified the conjecture of Alon, Babai and Suzuki.

In this paper, we will prove that if  $n\geq 2s-2r+1$ or $n\geq s+\max_{1\leq i\leq r}k_i$, then
\[
|\A|\leq{n-1\choose s}+{n-1\choose s-1}+\cdots+{n-1\choose s-2r+1}.
\]
This result strengthens the upper bound of Alon, Babai and Suzuki's conjecture when $n\geq 2s-2$.
\end{abstract}

\section{Introduction}
A family $\A$ of subsets of $[n]$ is called {\it intersecting} if every pair of distinct subsets $A_i, A_j\in \A$ have a nonempty intersection.
Let $L$ be a set of $s$ nonnegative integers. A family $\A$ of subsets of $[n]=\{1,2,\ldots,n\}$ is {\it $L$-intersecting} if $|A_i\cap A_j|\in L$ for every pair of distinct subsets $A_i,A_j\in\A$. A family $\A$ is {\it $k$-uniform} if it is a collection of $k$-subsets of $[n]$.  Thus, a $k$-uniform intersecting family is $L$-intersecting for $L=\{1,2,\ldots,k-1\}$.

The following is an intersection theorem of de Bruijin and Erd\"{o}s \cite{BE}.

\begin{thm}[de Bruijin and Erd\"{o}s, 1948 \cite{BE}] If $\A$ is a family of subsets of $[n]$ satisfying $|A_i \cap A_i| = 1$ for every pair of distinct subsets $A_i, A_j\in \A$,
then $|\A|\leq n$.
\end{thm}

A year later, Bose \cite{Bose} obtained the following more general intersection theorem which requires the
intersections to have exactly $\lambda$ elements.

\begin{thm}[Bose, 1949 \cite{Bose}] If $\A$ is a family of subsets of $[n]$ satisfying $|A_i \cap A_i| = \lambda$ for every pair of distinct subsets $A_i, A_j\in \A$,
then $|\A|\leq n$.
\end{thm}

In 1961, Erd\"{o}s, Ko and Rado \cite{EKR} proved the following classical result on $k$-uniform intersecting families.

\begin{thm}[Erd\"{o}s, Ko and Rado, 1961 \cite{EKR}]
Let $n \geq 2k$ and let $\A$ be a $k$-uniform intersecting family of subsets of $[n]$. Then $|\A| \leq {n-1\choose k-1}$ with
equality only when $\A$ consists of all $k$-subsets containing a common element.
\end{thm}

In 1975, Ray-Chaudhuri and Wilson \cite{RW} made a major progress by deriving the following upper
bound for a $k$-uniform $L$-intersecting family.

\begin{thm}[Ray-Chaudhuri and Wilson, 1975 \cite{RW}]\label{RW}
If $\A$ is a $k$-uniform $L$-intersecting family of subsets of $[n]$, then $|\A|\leq {n \choose s}$.
\end{thm}
In terms of parameters $n$ and $s$, this inequality is best possible, as shown by the set of all $s$-subsets of $[n]$ with $L=\{0,1,\ldots,s-1\}$.

In 1981, Frankl and Wilson \cite{FW} obtained the following celebrated theorem which extends Theorem \ref{RW} by allowing different subset sizes.

\begin{thm}[Frankl and Wilson, 1981 \cite{FW}]\label{FW}
If $\A$ is an $L$-intersecting family of subsets of $[n]$, then $\A\leq{n \choose s}+{n \choose s-1}+\cdots +{n \choose 0}$.
\end{thm}

The upper bound in Theorem \ref{FW} is best possible, as demonstrated by the set of all subsets of size at most $s$ of $[n]$.

In the same paper, a modular version of Theorem \ref{RW} was also proved.

\begin{thm}[Frankl and Wilson, 1981 \cite{FW}]\label{FW2}
If $\A$ is a $k$-uniform family of subsets of $[n]$ such that $k\pmod{p}\notin L$ and $|A_i\cap A_j|\pmod{p}\in L$ for all $i\neq j$, then $|\A|\leq {n \choose s}$.
\end{thm}

In 1991, Alon, Babai and Suzuki \cite{ABS} proved the following theorem, which is a generalization of Theorem \ref{FW2} by replacing the condition of uniformity with the condition that the members of $\A$ have $r$ different sizes.

\begin{thm}[Alon, Babai and Suzuki, 1991 \cite{ABS}]\label{ABS}
Let $K=\{k_1,k_2,\ldots,k_r\}$ and $L=\{l_1,l_2,\ldots,l_s\}$ be two disjoint subsets of $\{0,1,\ldots,p-1\}$, where $p$ is a prime, and let $\A$ be a family of subsets of $[n]$ such that $|A_i|\pmod{p}\in K$ for all $A_i\in \A$ and $|A_i\cap A_j|\pmod{p}\in L$ for $i\neq j$. If $r(s-r+1)\leq p-1$ and $n\geq s +\max_{1\leq i\leq r}k_i$, then $|\A|\leq {n \choose s}+{n \choose s-1}+\cdots+{n \choose s-r+1}$.
\end{thm}

In the proof of Theorem \ref{ABS}, Alon, Babai and Suzuki used a very elegant linear algebra method together with their Lemma~3.6 which needs the condition $r(s-r+1)\leq p-1$ and $n\geq s +\max_{1\leq i\leq r}{k_i}$. They conjectured that the condition $r(s-r+1)\leq p-1$ in the statement of their theorem can be dropped off. However, their approach cannot work for this stronger claim. In an effort to prove the Alon-Babai-Suzuki's conjecture, Snevily \cite{S} obtained the following result.

\begin{thm}[Snevily, 1994 \cite{S}]\label{Snevily}Let $p$ be a prime and $K,L$ be two disjoint subsets of $\{0,1,\ldots,p-1\}$. Let $|L|=s$ and let $\A$ be a family of subsets of $[n]$ such that $|A_i|\pmod{p}\in K$ for all $A_i\in \A$ and $|A_i\cap A_j|\pmod{p}\in L$ for $i\neq j$. Then $|\A|\leq {n-1 \choose s}+{n-1 \choose s-1}+\cdots+{n-1 \choose 0}$.
\end{thm}

Since ${n-1\choose s}+{n-1\choose s-1}={n\choose s}$ and ${n \choose s-1}>\sum_{i=0}^{s-2} {n-1\choose i}$ when $n$ is sufficiently large, Theorem~\ref{Snevily} not only confirms the conjecture of Alon, Babai and Suzuki in many cases but also strengthens the upper bound of their theorem when $n$ is sufficiently large.

In 2000, Qian and Ray-Chaudhuri \cite{QR} developed a new linear algebra approach and proved the next theorem which shows that the same conclusion in Theorem \ref{ABS} holds if the two conditions $r(s-r+1)\leq p-1$ and $n\geq s+\max_{1\leq i\leq r}{k_i}$ are replaced by a single more relaxed condition $n\geq 2s-r$.

\begin{thm}[Qian and Ray-Chaudhuri, 2000 \cite{QR}]\label{QR}
Let $p$ be a prime and let $L=\{l_1,l_2,\ldots,l_s\}$ and $K=\{k_1,k_2,\ldots,k_r\}$ be two disjoint subsets of $\{0,1,\ldots,p-1\}$ such that $n\geq 2s-r$. Suppose that $\A$ is a family of subsets of $[n]$ such that $|A_i|\pmod{p}\in K$ for all $A_i\in \A$ and $|A_i\cap A_j|\pmod{p}\in L$ for every $i\neq j$. Then
$|\A|\leq {n \choose s}+{n \choose s-1}+\cdots+{n \choose s-r+1}$.
\end{thm}


Recently, Hwang and Kim \cite{HK} verified the conjecture of Alon, Babai and Suzuki.

\begin{thm}[Hwang and Kim, 2015 \cite{HK}]\label{HK}
Let $K=\{k_1,k_2,\ldots,k_r\}$ and $L=\{l_1,l_2,\ldots.l_s\}$ be two disjoint subsets of $\{0,1,\ldots,p-1\}$, where $p$ is a prime, and let $\A$ be a family of subsets of $[n]$ such that $|A_i|\pmod{p}\in K$ for all $A_i\in \A$ and $|A_i\cap A_j|\pmod{p}\in L$ for $i\neq j$. If $n\geq s +\max_{1\leq i\leq r}k_i$, then $|\A|\leq {n \choose s}+{n \choose s-1}+\cdots+{n \choose s-r+1}$.
\end{thm}

We note here that in some instances Alon, Babai and Suzuki's condition
holds but Qian and Ray-Chaudhuri's condition does not, while in some other instances the later condition holds but the former condition does not.

In \cite{CL}, Chen and Liu strengthened the upper bounds of Theorem~\ref{Snevily} under the condition $\min \{k_i\} > \max\{l_i\}$.

\begin{thm}[Chen and Liu, 2009 \cite{CL}]\label{ChenLiu}
Let $p$ be a prime and let $L=\{l_1,l_2,\ldots,l_s\}$ and $K=\{k_1,k_2,\ldots,k_r\}$ be two disjoint subsets of $\{0,1,\ldots,p-1\}$ such that $\min \{k_i\} > \max\{l_i\}$. Suppose that $\A$ is a family of subsets of $[n]$ such that $|A_i|\pmod{p}\in K$ for all $A_i\in \A$ and $|A_i\cap A_j|\pmod{p}\in L$ for every $i\neq j$. Then
$|\A|\leq {n-1 \choose s}+{n-1 \choose s-1}+\cdots+{n-1 \choose s-2r+1}$.
\end{thm}

In \cite{LY}, Liu and Yang generalized Theorem~\ref{ChenLiu} under a relaxed condition $k_i > s-r$ for every $i$.

\begin{thm}[Liu and Yang, 2014 \cite{CL}]\label{LiuYang1} Let $p$ be a prime and let $L=\{l_1,l_2,\ldots,l_s\}$ and $K=\{k_1,k_2,\ldots,k_r\}$ be two disjoint subsets of $\{0,1,\ldots,p-1\}$ such that $k_i > s-r$ for every $i$. Suppose that $\A$ is a family of subsets of $[n]$ such that $|A_i|\pmod{p}\in K$ for all $A_i\in \A$ and $|A_i\cap A_j|\pmod{p}\in L$ for every $i\neq j$. Then $|\A|\leq {n-1 \choose s}+{n-1 \choose s-1}+\cdots+{n-1 \choose s-2r+1}$.
\end{thm}

In the same paper, they also obtained the same bound under the condition of Theorem~\ref{ABS}.

\begin{thm}[Liu and Yang, 2014 \cite{CL}]\label{LiuYang2} Let $p$ be a prime and let $L=\{l_1,l_2,\ldots,l_s\}$ and $K=\{k_1,k_2,\ldots,k_r\}$ be two disjoint subsets of $\{0,1,\ldots,p-1\}$ such that  $r(s-r+1)\leq p-1$ and $n\geq s +\max_{1\leq i\leq r}k_i$. Suppose that $\A$ is a family of subsets of $[n]$ such that $|A_i|\pmod{p}\in K$ for all $A_i\in \A$ and $|A_i\cap A_j|\pmod{p}\in L$ for every $i\neq j$. Then $|\A|\leq {n-1 \choose s}+{n-1 \choose s-1}+\cdots+{n-1 \choose s-2r+1}$.
\end{thm}

In this paper, we show that Theorem~\ref{LiuYang2} still holds under the Alon, Babai and Suzuki's condition; that is to say, we can drop  the condition $r(s-r+1)\leq p-1$ in Theorem~\ref{LiuYang2}.

\begin{thm}\label{cor}
Let $p$ be a prime and let $L=\{l_1,l_2,\ldots,l_s\}$ and $K=\{k_1,k_2,\ldots,k_r\}$ be two disjoint subsets of $\{0,1,\ldots,p-1\}$. Suppose that $\A$ is a family of subsets of $[n]$ such that $|A_i|\pmod{p}\in K$ for all $A_i\in \A$ and $|A_i\cap A_j|\pmod{p}\in L$ for every $i\neq j$.  If $n\geq s +\max_{1\leq i\leq r}k_i$, then $|\A|\leq {n-1 \choose s}+{n-1 \choose s-1}+\cdots+{n-1 \choose s-2r+1}$.
\end{thm}
Note that $ {n-1 \choose s}+{n-1 \choose s-1}+\cdots+{n-1 \choose s-2r+1}=  {n \choose s}+{n \choose s-2}+\cdots+{n \choose s-2(r-1)}$ and ${n \choose s-2i} < {n \choose s-i}$ for $1\leq i \leq r-1$ when $n\geq 2s-2$.  Our result strengthens the upper bound of Alon-Babai-Suzuki's conjecture (Theorems~\ref{HK}) when $n\geq 2s-2$.

In the proof of Theorem~\ref{cor}, we first prove that  the bound holds under the condition $n\geq 2s-2r+1$, which relaxes the condition $n\geq 2s-r$ in the theorem of  Qian and Ray-Chaudhuri.

\begin{thm}\label{main}
Let $p$ be a prime and let $L=\{l_1,l_2,\ldots,l_s\}$ and $K=\{k_1,k_2,\ldots,k_r\}$ be two disjoint subsets of $\{0,1,\ldots,p-1\}$. Suppose that $\A$ is a family of subsets of $[n]$ such that $|A_i|\pmod{p}\in K$ for all $A_i\in \A$ and $|A_i\cap A_j|\pmod{p}\in L$ for every $i\neq j$. If $n\geq 2s-2r+1$, then
$|\A|\leq {n-1 \choose s}+{n-1 \choose s-1}+\cdots+{n-1 \choose s-2r+1}$.
\end{thm}


Theorems \ref{ABS}, \ref{QR}, \ref{LiuYang1} and~\ref{LiuYang2} have been extended to $k$-wise $L$-intersecting families in \cite{GS,LY}. With a similar idea, our results can also be extended to the $k$-wise case.

\section{Proof of Theorem \ref{main}}
In this section we prove Theorem \ref{main}, which will be helpful in the proof of Theorem \ref{cor}.

 Throughout this section, let $X=[n-1]=\{1,2,\ldots,n-1\}$ be an $(n-1)$-element set, $p$ be a prime, and let $L=\{l_1,l_2,\ldots,l_s\}$ and $K=\{k_1,k_2,\ldots,k_r\}$ be two disjoint subsets of $\{0,1,\ldots,p-1\}$. Suppose that $\A=\{A_1,A_2,\ldots,A_m\}$ is a family of subsets of $[n]$ such that (1) $|A_i|\pmod{p}\in K$ for every $1\leq i\leq m$, (2) $|A_i\cap A_j|\pmod{p}\in L$ for $i\neq j$. Without loss of generality, assume that there exists a positive integer $t$ such that $n\notin A_i$ for $1\leq i\leq t$ and $n\in A_i$ for $i\geq t+1$. Denote
\[
\mathbb{P}_i(X)=\{S|S\subset X~and~|S|=i\}.
\]
We associate a variable $x_i$ for each $A_i\in \A$ and set $x=(x_1,x_2,\ldots,x_m)$. For each $I\subset X$, define
\[
L_{I}=\sum_{i:I\subset A_i\in \A}x_i.
\]
Consider the system of linear equation over the field $\mathbb{F}_p$:
\begin{eqnarray}\label{ls}
\{L_{I}=0, \textup{\ \ where~$I$~runs~through~}\cup_{i=0}^{s}\mathbb{P}_i(X)\}.
\end{eqnarray}

\begin{prop}\label{prop}
Assume that $L\cap K=\emptyset$. If $\A$ is a mod $p$ $L$-intersecting family with $|A_i|\pmod{p}\in K$ for every $i$, then the only solution of the above system of linear equations is the trivial solution.
\end{prop}

\begin{proof}
Let $v=(v_1,v_2,\ldots,v_m)$ be a solution to the system (\ref{ls}). We will show that $v$ is the zero solution over the field $\mathbb{F}_p$. Define $$g(x)=\prod_{j=1}^{s}(x-l_j),$$ and $$h(x)=g(x+1)=\prod_{j=1}^{s}(x+1-l_j).$$
Since ${x \choose 0},{x \choose 1},\ldots,{x \choose s}$ form a basis for the vector space spanned by all the polynomials in $\mathbb{F}_p[x]$ of degree at most $s$, there exist $a_0,a_1,\ldots,a_s\in \mathbb{F}_p$ and $b_0,b_1,\ldots,b_s\in \mathbb{F}_p$ such that
$$g(x)=\sum_{i=0}^{s}a_i{x \choose i},$$ and $$h(x)=\sum_{i=0}^{s}b_i{x \choose i}.$$
Let $A_{i_0}$ be an element in $\A$ with $v_{i_0}\neq 0$.
Next we prove the following identities:

If $n\notin A_{i_0}$, then
\begin{eqnarray}
\sum_{i=0}^{s}a_i\sum_{I\in \mathbb{P}_i(X),I\subset A_{i_0}}L_{I}=\sum_{A_i\in \A}g(|A_i\cap A_{i_0}|)x_i;
\end{eqnarray}

if $n\in A_{i_0}$, then
\begin{align}
\sum_{i=0}^{s}b_i\sum_{I\in \mathbb{P}_i(X),I\subset A_{i_0}}L_{I}&=\sum_{i=1}^{t}h(|A_i\cap A_{i_0}|)x_i+\sum_{i\geq t+1}h(|A_i\cap A_{i_0}|-1)x_i.
\end{align}
We prove them by comparing the coefficients of both sides. For any $A_i\in \A$, the coefficient of $x_{i}$ in the left hand side of $(2)$ is
\[
\sum_{i=0}^{s}a_i|\{I\in \mathbb{P}_i(X):I\subset A_{i_0},I\subset A_i\}|=\sum_{i=0}^{s}a_i{|A_i\cap A_{i_0}| \choose i},
\]
which is equal to $g(|A_i\cap A_{i_0}|)$ by the definition of $a_i$. This proves the identity $(2)$.

For any $i\leq t$, the coefficient of $x_i$ in the left hand side of $(3)$ is
\[
\sum_{i=0}^{s}b_i|\{I\in \mathbb{P}_i(X):I\subset A_{i_0},I\subset A_i\}|=\sum_{i=0}^{s}b_i{|A_i\cap A_{i_0}| \choose i},
\]
for any $i\geq t+1$, the coefficient of $x_i$ in the left hand side of $(3)$ is
\[
\sum_{i=0}^{s}b_i|\{I\in \mathbb{P}_i(X):I\subset A_{i_0},I\subset A_i\}|=\sum_{i=0}^{s}b_i{|A_i\cap A_{i_0}|-1 \choose i}.
\]
This proves the identity $(3)$.

If $n\not\in A_{i_0}$, substituting $x_i$ with $v_i$ for all $i$ in the identity $(2)$, we have
$$\sum_{i=0}^{s}a_i\sum_{I\in \mathbb{P}_i(X),I\subset A_{i_0}}L_{I}(v)=\sum_{A_i\in \A}g(|A_i\cap A_{i_0}|)v_i.$$
It is clear that the left hand side is $0$ since $v$ is a solution to (\ref{ls}). For $A_i\in \A$ with $i\ne i_0$, $|A_i\cap A_{i_0}|\pmod{p}\in L$ and so $g(|A_i\cap A_{i_0}|)=0$. Thus the right hand side of the above identity is equal to $g(|A_{i_0}|)v_{i_0}$. So $g(|A_{i_0}|)v_{i_0}=0$. Since $L\cap K=\emptyset$, we have $g(|A_{i_0}|)\ne 0$ and so $v_{i_0}=0$. This is a contradiction to the definition of $v$.

If $n\in A_{i_0}$, substituting $x_i$ with $v_i$ for all $i$ in the identity $(3)$, we have
\begin{align*}
\sum_{i=0}^{s}b_i\sum_{I\in \mathbb{P}_i(X),I\subset A_{i_0}}L_{I}(v)&=\sum_{i=1}^{t}h(|A_i\cap A_{i_0}|)v_i+\sum_{i\geq t+1}h(|A_i\cap A_{i_0}|-1)v_i\\
&=\sum_{i\geq t+1}h(|A_i\cap A_{i_0}|-1)v_i \textup{ \ \ since $v_i=0$ for all $i\leq t$}.
\end{align*}
Since $h(|A_i\cap A_{i_0}|-1)=g(|A_i\cap A_{i_0}|)$, with a similar argument to the above case, we can deduce the same
contradiction. Then the proposition follows.
\end{proof}

As a result of this proposition, we have:
\[
|\A|\leq \dim(\{L_{I}:I\in\cup_{i=0}^{s}\mathbb{P}_i(X)\}),
\]
where $\dim(\{L_{I}:I\in\cup_{i=0}^{s}\mathbb{P}_i(X)\})$ is defined to be the dimension of the space spanned by $\{L_{I}:I\in\cup_{i=0}^{s}\mathbb{P}_i(X)\}$.
In the remaining of this section, we make efforts to give an upper bound on this dimension.

\begin{lem}\label{cri}
For any $i\in\{0,1,\ldots,s-2r+1\}$ and every $I\in \mathbb{P}_i(X)$, the linear form $$\sum_{H\in \mathbb{P}_{i+2r}(X),I\subset H}L_{H}$$ is linearly dependent on the set of linear forms $\{L_{H}:i\leq |H|\leq i+2r-1,H\subset X\}$ over $\mathbb{F}_p$.
\end{lem}

\begin{proof}
Define $$f(x)=\left(\prod_{j=1}^{r}(x-(k_j-i))\right) \times \left(\prod_{j=1}^{r}(x-(k_j-1-i))\right).$$
We distinguish two cases.
\begin{enumerate}
  \item[(a)] $i\pmod{p}\notin K$ and $i+1\pmod{p}\notin K$ for all $i$. In this case $\forall \ k_j\in K$, $k_j-i\ne0$ and $k_j-i-1\ne0$ in $\mathbb{F}_p$ and so $c=(k_1-i)(k_2-i)\cdots(k_r-i)(k_1-i-1)\cdots(k_r-i-1)\ne0$ in $\mathbb{F}_p$. It is clear that there exist $a_1,a_2,\ldots,a_{2r-1}\in \mathbb{F}_p$, $a_{2r}=(2r)!\in\mathbb{F}_p-\{0\}$ such that $$a_1{x \choose 1}+a_2{x \choose 2}+\cdots+a_{2r}{x \choose 2r}=f(x)-c,$$
      since the polynomial in the right hand side has constant term equal to $0$.
      \par
      Next we show that
      \begin{eqnarray}\label{lem}
      \sum_{j=1}^{2r}a_j\sum_{H\in \mathbb{P}_{i+j}(X),I\subset H}L_{H}=-cL_{I}.
      \end{eqnarray}
      In fact both sides are linear forms in $x_{A}$, for $A\in \A$. The coefficient of $x_{A}$ in the left hand side is $\sum_{j=1}^{2r}a_j|\{H|I\subset H\subset A,n \not \in H, |H|=i+j\}|$. So it is equal to
      \begin{equation*}
        \begin{split}
                  & \begin{cases} 0, \textup{\ \ if $I\not\subset A$;} \\
                  a_1{|A|-i \choose 1}+a_2{|A|-i \choose 2}+\cdots+a_{2r}{|A|-i \choose 2r}, \textup{\ \ if $I\subset A$ and $n\notin A$;}\\
                  a_1{|A|-i-1 \choose 1}+a_2{|A|-i-1 \choose 2}+\cdots+a_{2r}{|A|-i-1 \choose 2r}, \textup{\ \ if $I\subset A$ and $n\in A$.}\\
                   \end{cases}
         \end{split}
       \end{equation*}
       By the above polynomial identity,
      \[
       \sum_{j=1}^{2r}a_j{|A|-i \choose j} =f(|A|-i)-c=-c \textup{\ \ since~}|A|\pmod{p}\in K;
      \]
      \[
      \sum_{j=1}^{2r} a_j{|A|-i-1 \choose j}=f(|A|-i-1)-c=-c \textup{\ \ since~}|A|\pmod{p}\in K.
      \]
      The coefficient of $x_{A}$ in the right hand side is obviously the same. This proves (\ref{lem}).

      Writing (\ref{lem}) in a different way, we have
      \[
      \sum_{H\in\mathbb{P}_{i+2r}(X),I\subset H}L_{H}=-\frac{1}{(2r)!}(cL_{I}+\sum_{j=1}^{2r-1}a_j\sum_{H\in\mathbb{P}_{i+j}(X),I\subset H}L_H).
      \]
      This proves the lemma in case (a).

  \item[(b)] $i\pmod{p}\in K$ or $i+1\pmod{p}\in K$ for some $i$. In this case, the constant term of $(x-(k_1-i))(x-(k_2-i))\cdots(x-(k_r-i))(x-(k_1-i-1))\cdots(x-(k_r-i-1))$ is $0\in \mathbb{F}_p$. So there exists $a_1,a_2,\ldots,a_{2r-1}\in \mathbb{F}_p$, $a_{2r}=(2r)!\in\mathbb{F}_p-\{0\}$ such that

       \[
      a_1{x \choose 1}+a_2{x \choose 2}+\cdots+a_{2r}{x \choose 2r}=f(x)
      \]
      As a consequence we have
      \[
      \sum_{j=1}^{2r}a_j\sum_{H\in \mathbb{P}_{i+j}(X),I\subset H}L_{H}=0~~\forall I\in\mathbb{P}_i(X),
      \]
      i.e. we have
      \[
      \sum_{H\in\mathbb{P}_{i+2r}(X),I\subset H}L_{H}=-\frac{1}{(2r)!}(\sum_{j=1}^{2r-1}a_j\sum_{H\in\mathbb{P}_{i+j}(X),I\subset H}L_H).
      \]
      This finishes the proof of this lemma.
\end{enumerate}
\end{proof}

\begin{cor}\label{cricor} With the same condition as in Lemma~\ref{cri}, we have
\begin{align*}\langle &L_H:H\in\cup_{j=i}^{i+2r-1}\mathbb{P}_j(X)\rangle\\
&=\left\langle L_H:H\in\cup_{j=i}^{i+2r-1}\mathbb{P}_j(X)\right\rangle+\left\langle\sum_{H\in\mathbb{P}_{i+2r}(X),I\subset H}L_H:I\in\mathbb{P}_i(X)\right\rangle
\end{align*}
Here ${\langle L_H:H\in\cup_{j=i}^{i+2r-1}\mathbb{P}_j(X)\rangle}$ is the vector space spanned by $\{ L_H:H\in\cup_{j=i}^{i+2r-1}\mathbb{P}_j(X)\}$.
\end{cor}

The rest of the proof is similar to the proof of Theorem~\ref{QR} given by Qian and Ray-Chaudhuri \cite{QR}.  The next lemma is a restatement of \cite[Lemma~2]{QR}, and is used to prove Lemma~\ref{recurbound}.

\begin{lem}\label{count}
For any positive integers $u,v$ with $u < v <p$ and $u+v \leq n-1$, we have
$$\dim\left(\frac{\langle L_J:J\in \mathbb{P}_v(X)\rangle}{\langle \sum_{J\in \mathbb{P}_v(X), I \subset J}L_J:I\in \mathbb{P}_u(X)\rangle}\right)\leq {n-1\choose v}-{n-1\choose u}.$$
Here $\frac{A}{B}$ is the quotient space of two vector spaces $A$ and $B$ with $B\leq A$.
\end{lem}

\begin{lem}\label{recurbound}
For any $i\in \{0,1,\ldots,s-2r+1\}$,
\[
{n-1 \choose i}+{n-1\choose i+1}+\cdots+{n-1\choose i+2r-1}+\dim\left(\frac{\langle L_H:H\in\cup_{j=i}^{s}\mathbb{P}_j(X)\rangle}{\langle L_H:H\in\cup_{j=i}^{i+2r-1}\mathbb{P}_j(X)\rangle}\right)
\]
\[
\leq{n-1\choose s-2r+1}+{n-1\choose s-2r+2}+\cdots+{n-1\choose s}.
\]
\end{lem}

\begin{proof}
We induct on $s-2r+1-i$. It is clearly true when $s-2r+1-i=0$. Suppose the lemma holds for $s-2r+1-i<l$ for some positive integer $l$. Now we want to show that it holds for $s-2r+1-i=l$.

We observe that $i+i+2r\leq(s-2r)+(s-2r)+2r\leq n-1$ by the condition in the theorem. By Corollary~\ref{cricor} and Lemma~\ref{count}, we have
\begin{align*}
\dim&\left(\frac{\langle L_H:H\in\cup_{j=i}^{i+2r}\mathbb{P}_j(X)\rangle}{\langle L_H:H\in\cup_{j=i}^{i+2r-1}\mathbb{P}_j(X)\rangle}\right)\\
&=\dim\left(\frac{\langle L_H:H\in \cup_{j=i}^{i+2r-1}\mathbb{P}_j(X)\rangle+\langle L_H:H\in \mathbb{P}_{i+2r}(X)\rangle}{\langle L_H:H\in\cup_{j=i}^{i+2r-1}\mathbb{P}_j(X)\rangle+\langle\sum_{H\in\mathbb{P}_{i+2r}(X),I\subset H}L_H:I\in\mathbb{P}_i(X)\rangle}\right) \\
&\leq \dim\left(\frac{L_H:H\in \mathbb{P}_{i+2r}(X)}{\sum_{H\in\mathbb{P}_{i+2r}(X),I\subset H}L_H:I\in\mathbb{P}_i(X)}\right)\\
&\leq{n-1\choose i+2r}-{n-1\choose i}.
\end{align*}
Now we are ready to prove the lemma.
\begin{align*}
{n-1 \choose i}&+{n-1\choose i+1}+\cdots+{n-1\choose i+2r-1}+\dim\left(\frac{\langle L_H:H\in\cup_{j=i}^{s}\mathbb{P}_j(X)\rangle}{\langle L_H:H\in\cup_{j=i}^{i+2r-1}\mathbb{P}_j(X)\rangle}\right) \\
=&{n-1 \choose i}+{n-1\choose i+1}+\cdots+{n-1\choose i+2r-1}+\dim\left(\frac{\langle L_H:H\in\cup_{j=i}^{i+2r}\mathbb{P}_j(X)\rangle}{\langle L_H:H\in\cup_{j=i}^{i+2r-1}\mathbb{P}_j(X)\rangle}\right)\\
&+\dim\left(\frac{\langle L_H:H\in\cup_{j=i}^{s}\mathbb{P}_j(X)\rangle}{\langle L_H:H\in\cup_{j=i}^{i+2r}\mathbb{P}_j(X)\rangle}\right)\\
=&{n-1 \choose i}+{n-1\choose i+1}+\cdots+{n-1\choose i+2r-1}+\dim\left(\frac{\langle L_H:H\in\cup_{j=i}^{i+2r}\mathbb{P}_j(X)\rangle}{\langle L_H:H\in\cup_{j=i}^{i+2r-1}\mathbb{P}_j(X)\rangle}\right)\\
&+\dim\left(\frac{\langle L_H:H\in\mathbb{P}_i(X)\rangle+\langle L_H:H\in\cup_{j=i+1}^{s}\mathbb{P}_j(X)\rangle}{\langle L_H:H\in\mathbb{P}_i(X)\rangle+\langle L_H:H\in\cup_{j=i+1}^{i+2r}\mathbb{P}_j(X)\rangle}\right) \\
\leq&{n-1 \choose i}+{n-1\choose i+1}+\cdots+{n-1\choose i+2r-1}+\dim\left(\frac{\langle L_H:H\in\cup_{j=i}^{i+2r}\mathbb{P}_j(X)\rangle}{\langle L_H:H\in\cup_{j=i}^{i+2r-1}\mathbb{P}_j(X)\rangle}\right)\\
&+\dim\left(\frac{\langle L_H:H\in\cup_{j=i+1}^{s}\mathbb{P}_j(X)\rangle}{\langle L_H:H\in\cup_{j=i+1}^{i+2r}\mathbb{P}_j(X)\rangle}\right) \\
\leq&{n-1 \choose i}+{n-1\choose i+1}+\cdots+{n-1\choose i+2r-1}+{n-1\choose i+2r}-{n-1\choose i}\\
&+\dim\left(\frac{\langle L_H:H\in\cup_{j=i+1}^{s}\mathbb{P}_j(X)\rangle}{\langle L_H:H\in\cup_{j=i+1}^{i+2r}\mathbb{P}_j(X)\rangle}\right)\\
=&{n-1\choose i+1}+\cdots+{n\choose i+2r}+\dim\left(\frac{\langle L_H:H\in\cup_{j=i+1}^{s}\mathbb{P}_j(X)\rangle}{\langle L_H:H\in\cup_{j=i+1}^{i+2r}\mathbb{P}_j(X)\rangle}\right)\\
\leq&{n-1\choose s-2r+1}+\cdots+{n-1\choose s},
\end{align*}
where the last step follows from the induction hypothesis since $s-2r+1-(i+1)<l$.
\end{proof}

We are now turning to the proof of Theorem \ref{main}.
\begin{proof}
\begin{align*}
|\A|\leq &\dim(\langle L_H:H\in\cup_{i=0}^{s}\mathbb{P}_i(X)\rangle)\\
\leq &\dim(\langle L_H:H\in\cup_{i=0}^{2r-1}\mathbb{P}_i(X)\rangle)+\dim\left(\frac{\langle L_H:H\in\cup_{i=0}^{s}\mathbb{P}_j(X)\rangle}{\langle L_H:H\in\cup_{i=0}^{2r-1}\mathbb{P}_j(X)\rangle}\right)\\
\leq &{n-1\choose 0}+{n-1\choose 1}+\cdots+{n-1\choose 2r-1}+\dim\left(\frac{\langle L_H:H\in\cup_{i=0}^{s}\mathbb{P}_j(X)\rangle}{\langle L_H:H\in\cup_{i=0}^{2r-1}\mathbb{P}_j(X)\rangle}\right)\\
\leq &{n-1\choose s-2r+1}+{n-1\choose s-2r+2}+\cdots+{n-1\choose s} \textup{\ \ by taking $i=0$ in Lemma~\ref{recurbound}},
\end{align*}
which completes the proof of the theorem.
\end{proof}

\section{Proof of Theorem \ref{cor}}
Throughout this section, we let $p$ be a prime and we will use $x=(x_1,x_2,\ldots,x_n)$ to denote a vector of $n$ variables with each variable $x_i$ taking values $0$ or $1$. A polynomial $f(x)$ in $n$ variables $x_i$, for $1\leq i\leq n$, is called {\it multilinear} if the power of each variable $x_i$ in each term is at most one. Clearly, if each variable $x_i$ only takes the values $0$ or $1$, then any polynomial in variable $x$ can be regarded as multilinear. For a subset $A$ of $[n]$, we define the incidence vector $v_{A}$ of $A$ to be the vector $v=(v_1,v_2,\ldots,v_n)$ with $v_i=1$ if $i\in A$ and $v_i=0$ otherwise.

Let $L=\{l_1,l_2,\ldots,l_s\}$ and $K=\{k_1,k_2,\ldots,k_r\}$ be two disjoint subsets of $\{0,1,\ldots,p-1\}$,  where the elements of $K$ are arranged in increasing order. Suppose that $\A=\{A_1,\ldots,A_m\}$ is the family of subsets of $[n]$ satisfying the conditions in Theorem \ref{cor}.
Without loss of generality, we may assume that $n\in A_j$ for $j\geq t+1$ and $n\notin A_j$ for $1\leq j\leq t$.

For each $A_j\in\A$, define
\[
f_{A_j}(x)=\prod_{i=1}^{s}(v_{A_j}x-l_i),
\]
where $x=(x_1,x_2,\ldots,x_n)$ is a vector of $n$ variables with each variable $x_i$ taking values $0$ or $1$. Then each $f_{A_j}(x)$ is a multilinear polynomial of degree at most $s$.

Let $Q$ be the family of subsets of $[n-1]$ with sizes at most $s-1$. Then $|Q|=\sum_{i=0}^{s-1}{n-1\choose i}$. For each $L\in Q$, define
\[
q_{L}(x)=(1-x_n)\prod_{i\in L}x_i.
\]
Then each $q_L(x)$ is a multilinear polynomial of degree at most $s$.

Denote $K-1=\{k_i-1 |k_i\in K\}$. Then $|K\cup(K-1)|\leq 2r$. Set
\[
g(x)=\prod_{h\in K\cup(K-1)}\left(\sum_{i=1}^{n-1}x_i-h\right).
\]
Let $W$ be the family of subsets of $[n-1]$ with sizes at most $s-2r$. Then $|W|=\sum_{i=0}^{s-2r}{n-1\choose i}$. For each $I\in W$, define
\[
g_I(x)=g(x)\prod_{i\in I}x_i.
\]
Then each $g_I(x)$ is a multilinear polynomial of degree at most $s$.

We want to show that the polynomials in
\[
\{f_{A_i(x)}|1\leq i\leq m\}\cup\{q_L(x)|L\in Q\}\cup\{g_I(x)|I\in W\}
\]
are linearly independent over the field $\mathbb{F}_p$. Suppose that we have a linear combination of these polynomials that equals $0$:
\begin{eqnarray}\label{linear}
\sum_{i=1}^{m}a_if_{A_i}(x)+\sum_{L\in Q}b_Lq_L(x)+\sum_{I\in W}u_Ig_I(x)=0,
\end{eqnarray}
with all coefficients $a_i,b_L$ and $u_I$ being in ${\mathbb F}_p$.

\textbf{Claim $1$.} $a_i=0$ for each $i$ with $n\in A_i$.

Suppose, to the contrary, that $i_0$ is a subscript such that $n\in A_{i_0}$ and $a_{i_0}\ne 0$. Since $n\in A_{i_0}$, $q_{L}(v_{A_{i_0}})=0$ for every $L\in Q$. Recall that $f_{A_j}(v_{i_0})=0$ for $j\ne i_0$ and $g(v_{i_0})=0$. By evaluating (\ref{linear}) with $x=v_{A_{i_0}}$, we obtain that $a_{i_0}f_{A_{i_0}}(v_{A_{i_0}})=0\pmod{p}$. Since $f_{A_{i_0}}(v_{A_{i_0}})\ne 0$, we have $a_{i_0}=0$, a contradiction. Thus, Claim 1 holds.

\textbf{Claim $2$.} $a_i=0$ for each $i$ with $n\not \in A_i$. Applying Claim 1, we get
\begin{eqnarray}\label{linear2}
\sum_{i=1}^{t}a_if_{A_i}(x)+\sum_{L\in Q}b_Lq_L(x)+\sum_{I\in W}u_Ig_I(x)=0.
\end{eqnarray}

Suppose, to the contrary, that $i_0$ is a subscript such that $n\notin A_{i_0}$ and $a_{i_0}\ne 0$. Let $v_{i_0}'=v_{i_0}+(0,0,\ldots,0,1)$. Then $q_L(v_{i_0}')=0$ for every $L\in Q$. Note that $f_{A_j}(v_{i_0}')=f_{A_j}(v_{i_0})$ for each $j$ with $n\notin A_j$ and $g(v_{i_0}')=0$. By evaluating (\ref{linear2}) with $x=v_{i_0}'$, we obtain $a_{i_0}f_{A_{i_0}}(v_{i_0}')=a_{i_0}f_{A_{i_0}}(v_{i_0})=0\pmod{p}$ which implies $a_{i_0}=0$, a contradiction. Thus, the claim is verified.

\textbf{Claim $3$.} $b_L=0$ for each $L\in Q$.

By Claims $1$ and $2$, we obtain
\begin{eqnarray}\label{linear3}
\sum_{L\in Q}b_Lq_L(x)+\sum_{I\in W}u_Ig_I(x)=0.
\end{eqnarray}
Set $x_n=0$ in (\ref{linear3}), then
\[
\sum_{L\in Q}b_L\prod_{i\in L}x_i+\sum_{I\in W}u_Ig_I(x)=0.
\]
Subtracting the above equality from (\ref{linear3}), we get
\[
\sum_{L\in Q}b_L\left(x_n\prod_{i\in L}x_i\right)=0.
\]
Setting $x_n=1$, we obtain
\[
\sum_{L\in Q}b_L\prod_{i\in L}x_i=0.
\]
It is not difficult to see that the polynomials $\prod_{i\in L}x_i$, $L\in Q$, are linearly independent. Therefore, we conclude that $b_L=0$ for each $L\in Q$.

By Claims 1-3, we now have
\[
\sum_{I\in W}u_Ig_I(x)=0.
\]
Thus it is sufficient to prove $g_I$'s are linearly independent.

Let $N$ be a positive integer and $H=\{h_1, h_2, \ldots, h_u\}$ be a subset of $[N]$ with all the elements being arranged in increasing order.
 We say $H$ {\it has a  gap} of size $\geq g$ if either $h_1\geq g-1, N-h_u\geq g-1$, or $h_{i+1}-h_{i}\geq g$ for some $i$ ($1\leq i\leq u-1$).  The following result obtained by Alon, Babai and Suzuki \cite{ABS} is critical to our proof.

\begin{lem}\label{LEM}
Let $H$ be a subset of $\{0,1,\ldots,p-1\}$. Let $p(x)$ denote the polynomial function defined by $p(x)=\prod_{h\in H}(x_1+x_2+\cdots+x_N-h)$. If the set $(H+p\mathbb{Z})\cap [N]$ has a gap $\geq g+1$, where $g$ is a positive integer, then the set of polynomials $\{p_I(x):|I|\leq g-1, I\in {N}\}$ is linearly independent over $\mathbb{F}_p$, where $p_I(x)=p(x)\prod_{i\in I}x_i$.
\end{lem}

To apply Lemma~\ref{LEM}, we define the set $H$ as follows: $H=(K\cup(K-1)+p\mathbb{Z})\cap[n-1]$. We can divide $n-1$ into the the following four cases:
\begin{enumerate}
  \item $s+k_r-1\leq n-1<p+k_1-1$;
  \item $s+k_r-1<p+k_1-1\leq n-1$;
  \item $(s-2r+1)+k_r<p+k_1-1\leq s+k_r-1\leq n-1$;
  \item $p+k_1-1\leq (s-2r+1)+k_r\leq s+k_r-1\leq n-1$.
\end{enumerate}

\noindent {\bf Case 1:} $s+k_r-1\leq n-1<p+k_1-1$.

Since $n-1<p+k_1-1$, the set $H$ consists of only $\{k_1-1,k_1,\ldots,k_r\}$. From $s+k_r-1\leq n-1$, we obtain $n-1-k_r\geq s-1 \geq s-2r+1$. By the definition of the gap, $H$ has a gap $\geq s-2r+2$.

\noindent {\bf Case 2:} $s+k_r-1<p+k_1-1\leq n-1$.

Since $n-1\geq p+k_1-1$, the set $H$ contains at least the following elements $\{k_1-1,k_1,\ldots,k_r,p+k_1-1\}$. From $s+k_r-1<p+k_1-1$, we derive $(p+k_1-1)-k_r\geq s\geq s-2r+2$. Thus, $H$ has a gap $\geq s-2r+2$.

\noindent {\bf Case 3:} $(s-2r+1)+k_r<p+k_1-1\leq s+k_r-1\leq n-1$.

Since $n-1\geq p+k_1-1$, $H$ contains at least the following elements $\{k_1-1,k_1,\ldots,k_r,p+k_1-1\}$. Since $(s-2r+1)+k_r<p+k_1-1$, we have $(p+k_1-1)-k_r>s-2r+1$. Then $H$ has a gap $\geq s-2r+2$.

By applying Lemma \ref{LEM}, we conclude that the set of polynomials $\{g_I(x):I\in W\}$ is linearly independent over $\mathbb{F}_p$, and so $u_I=0$ for each $I\in W$.

In summary, for the Cases $1$--$3$, we have shown that the polynomials in
\[
\{f_{A_i(x)}|1\leq i\leq m\}\cup\{q_L(x)|L\in Q\}\cup\{g_I(x)|I\in W\}
\]
are linearly independent over the field $\mathbb{F}_p$. Since the set of all monomials in variables $x_1, x_2,\ldots, x_n$ of degree at most $s$ forms a basis for the vector space of multilinear polynomials of degree at most $s$, it follows that
\[
|\A|+\sum_{i=0}^{s-1}{n-1\choose i}+\sum_{i=0}^{s-2r}{n-1\choose i}\leq\sum_{i=0}^{s}{n\choose i},
\]
which implies that
\[
|\A|\leq{n-1\choose s}+{n-1\choose s-1}+\cdots+{n-1\choose s-2r+1}.
\]
This completes the proof of the theorem for the Cases $1$--$3$.

Since Theorem \ref{main} has shown that the statement of Theorem \ref{cor} remains true under the condition $n\geq 2s-2r+1$, we just consider $n\leq 2s-2r$ for the Case 4. The following argument is similar to the technique Hwang and Kim used for the proof of Alon-Babai-Suzuki's conjecture.

Since $p+k_1-1\leq(s-2r+1)+k_r\leq s+k_r-1\leq n-1\leq2s-2r-1$, we obtain $k_r\leq s-2r$. Thus, we have $r+s\leq p \leq s-2r+2+k_r-k_1\leq 2s-4r+1$. This implies $s\geq 5r-1$. Since $n\leq 2s-2r<2p$, we have $|A_i|\in (K+p\mathbb{Z})\cap [n]=\{k_1,k_2,\ldots,k_r,p+k_1,\ldots,p+k_{c}\}$ for some  $1\leq c\leq r$. This gives
\[
|\A|\leq{n\choose k_1}+{n\choose k_2}+\cdots+{n\choose k_r}+{n\choose p+k_1}+\cdots+{n\choose p+k_c}.
\]
We will show that the right hand side of the above inequality is less than or equal to ${n-1\choose s}+{n-1\choose s-1}+\ldots+{n-1\choose s-2r+1}={n\choose s}+{n\choose s-2}+\ldots+{n\choose s-2r+2}$. Since $s+r+k_1-1\leq p+k_1-1\leq (s-2r+1)+k_r$, we have $k_r\geq 3r-2+k_1$. Let $n=2s-2r-\delta$ for integer $\delta$, where $0\leq \delta\leq s-5r+1$, since $2s-2r\geq n\geq s+k_r\geq s+3r-2+k_1$. Since the sequence $\{{n\choose k}\}$ is unimodal and symmetric around $n/2$, we have $|s-n/2|=r+\delta/2>r-\delta/2-2=|n/2-(s-2r+2)|$.

Therefore we have
\begin{eqnarray}\label{min}
\min\left[{n\choose s},{n\choose s-2},\ldots,{n\choose s-2r+2}\right]={n\choose s}.
\end{eqnarray}

Since $n=2s-2r-\delta\geq p+k_c\geq r+s+k_c$, we have $k_c\leq s-3r-\delta$. For $1\leq i\leq c$, $k_i$ can be written as $k_i=s-3r-\delta-a_i$, where $0< a_i\leq s-3r-\delta$. Thus, we have $p+k_i\geq r+s+k_i= 2s-2r-\delta-a_i$ where $1\leq i\leq c$. Since $2s-2r-\delta-a_i\geq s+r>n/2$, we have
\[
\sum_{i=1}^{c}\left({n\choose k_i}+{n\choose p+k_i}\right)\leq\sum_{i=1}^{c}\left({n\choose s-3r-\delta-a_i}+{n\choose 2s-2r-\delta-a_i}\right).
\]

For $c+1\leq i\leq r$, we derive $k_i\leq k_r< s-2r-\delta<n/2$. Noting that $|s-n/2|=r+\delta/2=|n/2-(s-2r-\delta)|$, we have ${n\choose k_i}\leq{n\choose s}$ for all $c+1\leq i\leq r$. Then
\begin{align*}
|\A|&\leq\sum_{i=1}^{c}\left({n\choose k_i}+{n\choose p+k_i}\right)+\sum_{i=c+1}^{r}{n\choose k_i}\\
&\leq\sum_{i=1}^{c}\left({n\choose s-3r-\delta-a_i}+{n\choose 2s-2r-\delta-a_i}\right)+(r-c){n\choose s}.
\end{align*}
With the help of the next lemma, we can complete our proof.

\begin{lem}{\rm \cite{HK}}\label{Kimequality}
For all $0\leq c<k\leq n/2$, we have
\[
{n\choose k-1-c}+{n\choose c}\leq{n\choose k}.
\]
\end{lem}

Let $k=n-s=s-2r-\delta<n/2$, apply Lemma~\ref{Kimequality}.
For every $0\leq a\leq s-3r-\delta<k$, we have
\begin{align*}
{n\choose s-3r-\delta-a}&+{n\choose 2s-2r-\delta-a}\\
=&{n\choose n-s-r-a}+{n\choose n-a}\\
=&{n\choose k-r-a}+{n\choose a}\\
\leq&{n\choose k-1-a}+{n\choose a}\\
\leq&{n\choose k}={n\choose s}.
\end{align*}

We now finish the proof of Theorem \ref{cor} for the Case 4.
\[
|\A|\leq\sum_{i=1}^{c}\left({n\choose s-3r-\delta-a_i}+{n\choose 2s-2r-\delta-a_i}\right)+(r-c){n\choose s}\leq r{n\choose s}.
\]
By (\ref{min}), we have
\[
|\A|\leq{n\choose s}+{n\choose s-2}+\cdots+{n\choose s-2r+2}={n-1\choose s}+{n-1\choose s-1}+\cdots+{n-1\choose s-2r+1}.
\]

\subsection*{Acknowledgements}

The research of G. Ge was supported by the National Natural Science Foundation of China under Grant Nos. 61171198, 11431003 and 61571310, and the
Importation and Development of High-Caliber Talents Project of Beijing Municipal Institutions.


\begin{thebibliography}{10}

\bibitem{ABS}
N.~Alon, L.~Babai, and H.~Suzuki.
\newblock Multilinear polynomials and {F}rankl--{R}ay-{C}haudhuri--{W}ilson
  type intersection theorems.
\newblock {\em J. Combin. Theory Ser. A}, 58(2):165--180, 1991.

\bibitem{Bose}
R.~C. Bose.
\newblock A note on {F}isher's inequality for balanced incomplete block
  designs.
\newblock {\em Ann. Math. Statistics}, 20:619--620, 1949.

\bibitem{CL}
W.~Y.~C. Chen and J.~Liu.
\newblock Set systems with {$L$}-intersections modulo a prime number.
\newblock {\em J. Combin. Theory Ser. A}, 116(1):120--131, 2009.

\bibitem{BE}
N.~G. de~Bruijn and P.~Erd{\"o}s.
\newblock On a combinatorial problem.
\newblock {\em Nederl. Akad. Wetensch., Proc.}, 51:1277--1279 = Indagationes
  Math. 10, 421--423 (1948), 1948.

\bibitem{EKR}
P.~Erd{\H{o}}s, C.~Ko, and R.~Rado.
\newblock Intersection theorems for systems of finite sets.
\newblock {\em Quart. J. Math. Oxford Ser. (2)}, 12:313--320, 1961.

\bibitem{FW}
P.~Frankl and R.~M. Wilson.
\newblock Intersection theorems with geometric consequences.
\newblock {\em Combinatorica}, 1(4):357--368, 1981.

\bibitem{GS}
V.~Grolmusz and B.~Sudakov.
\newblock On {$k$}-wise set-intersections and {$k$}-wise {H}amming-distances.
\newblock {\em J. Combin. Theory Ser. A}, 99(1):180--190, 2002.

\bibitem{HK}
K.-W. Hwang and Y.~Kim.
\newblock A proof of {A}lon-{B}abai-{S}uzuki's conjecture and multilinear
  polynomials.
\newblock {\em European J. Combin.}, 43:289--294, 2015.

\bibitem{LY}
J.~Liu and W.~Yang.
\newblock Set systems with restricted {$k$}-wise {$L$}-intersections modulo a
  prime number.
\newblock {\em European J. Combin.}, 36:707--719, 2014.

\bibitem{QR}
J.~Qian and D.~K. Ray-Chaudhuri.
\newblock On mod-{$p$} {A}lon-{B}abai-{S}uzuki inequality.
\newblock {\em J. Algebraic Combin.}, 12(1):85--93, 2000.

\bibitem{RW}
D.~K. Ray-Chaudhuri and R.~M. Wilson.
\newblock On {$t$}-designs.
\newblock {\em Osaka J. Math.}, 12(3):737--744, 1975.

\bibitem{S}
H.~S. Snevily.
\newblock On generalizations of the de {B}ruijn-{E}rd{\H o}s theorem.
\newblock {\em J. Combin. Theory Ser. A}, 68(1):232--238, 1994.

\end{thebibliography}
\end{document}